\documentclass[a4paper,11pt]{amsart}

\usepackage{amscd,amsmath,amssymb,fancyhdr,color}
\usepackage[utf8x]{inputenc}
\usepackage{amsfonts}
\usepackage{amsthm}
\usepackage{accents}
\usepackage{graphicx}
\usepackage{float}
\usepackage{subcaption}
\usepackage{verbatim}
\usepackage{indentfirst}
\usepackage{dsfont}
\usepackage{tikz}
\usetikzlibrary{matrix}
\usepackage{enumerate}
\usepackage{extarrows}

\usepackage[top=1.5in, bottom=1in, left=0.9in, right=0.9in]{geometry}

\usepackage{scalerel,stackengine}
\stackMath
\newcommand\reallywidehat[1]{%
	\savestack{\tmpbox}{\stretchto{%
			\scaleto{%
				\scalerel*[\widthof{\ensuremath{#1}}]{\kern-.6pt\bigwedge\kern-.6pt}%
				{\rule[-\textheight/2]{1ex}{\textheight}}
			}{\textheight}%
		}{0.5ex}}%
	\stackon[1pt]{#1}{\tmpbox}%
}

\usepackage{faktor}

\usepackage[backref=page]{hyperref}
\renewcommand*{\backref}[1]{}
\renewcommand*{\backrefalt}[4]{%
	\ifcase #1 (Not cited.)%
	\or        (Cited on page~#2.)%
	\else      (Cited on pages~#2.)%
	\fi}

\hypersetup{
	colorlinks   = true,
	citecolor    = magenta
}

\newcommand{\K}{K\"ahler}

\DeclareMathOperator{\dvol}{\textit{d}vol}

\numberwithin{equation}{section}

\def\eqref#1{(\ref{#1})}

\newcommand{\C}{{\mathbb C}}
\newcommand{\R}{{\mathbb R}}

\def\1{\sqrt{-1}\:}

\newcommand{\cntrct}                
{\hspace{2pt}\raisebox{1pt}{\text{$\lrcorner$}}\hspace{2pt}}

\renewcommand{\dim}{\operatorname{dim}}

\renewcommand{\Re}{\operatorname{Re}}

\newcommand{\ie}{{\em i.e. }}
\newcommand{\eg}{{\em {\em e.g.} }}

\newcounter{Mycounter}[section]
\newcounter{lemma}[section]
\newcounter{claim}[section]

\newcounter{sublemma}[section]

\newcounter{corollary}[section]

\newcounter{theorem}[section]

\newcounter{conjecture}[section]

\newcounter{proposition}[section]

\newcounter{definition}[section]

\newcounter{example}[section]

\newcounter{remark}[section]

\newcounter{problem}[section]

\newcounter{question}[section]

\makeatletter

\@addtoreset{equation}{section}

\@addtoreset{footnote}{section}

\makeatother

\def\blacksquare{\hbox{\vrule width 5pt height 5pt depth 0pt}}

\usetikzlibrary{arrows,chains,matrix,positioning,scopes}

\makeatletter
\tikzset{join/.code=\tikzset{after node path={%
			\ifx\tikzchainprevious\pgfutil@empty\else(\tikzchainprevious)%
			edge[every join]#1(\tikzchaincurrent)\fi}}}
\makeatother

\makeatletter
\def\blfootnote{\gdef\@thefnmark{}\@footnotetext}
\makeatother

\tikzset{>=stealth',every on chain/.append style={join},
	every join/.style={->}}

\dedicatory{
	Dedicated to Paul Gauduchon for his 75th birthday.\\
	Bon anniversaire!
}

\begin{document}

	\title[Compatibility between non-K\" ahler structures]{Compatibility between non-K\" ahler structures on complex (nil)manifolds}
	
	\author{Liviu Ornea}
	\address{Liviu Ornea \newline
		\textsc{\indent University of Bucharest, Faculty of Mathematics and Computer Science \newline
			\indent 14 Academiei Str., Bucharest, Romania}\newline
		\indent	{\em and}\newline
		\indent Institute of Mathematics ``Simion Stoilow'' of the Romanian Academy\newline 
		\indent 21 Calea Grivitei Street, 010702, Bucharest, Romania}
	\email{lornea@fmi.unibuc.ro, liviu.ornea@imar.ro}
	
	\author{Alexandra Otiman}
	\address{Alexandra Otiman \newline 
		\indent Dipartimento di Matematica e Informatica ``Ulisse Dini", \newline
		\indent Universit\' a degli studi di Firenze, \newline
		\indent viale Morgagni 67/A, 50134, Firenze, Italy \newline
		\indent	{\em and}\newline
		\indent Institute of Mathematics ``Simion Stoilow'' of the Romanian Academy\newline 
		\indent 21 Calea Grivitei Street, 010702, Bucharest, Romania}
	\email{alexandraiulia.otiman@unifi.it, alexandra.otiman@imar.ro}
	
	\author{Miron Stanciu}
	\address{Miron Stanciu\newline 
		\textsc{\indent University of Bucharest, Faculty of Mathematics and Computer Science \newline
			\indent 14 Academiei Str., Bucharest, Romania}\newline
		\indent	{\em and}\newline
		\indent Institute of Mathematics ``Simion Stoilow'' of the Romanian Academy\newline 
		\indent 21 Calea Grivitei Street, 010702, Bucharest, Romania}
	\email{miron.stanciu@fmi.unibuc.ro; miron.stanciu@imar.ro}

	\begin{abstract}
		We study the interplay between the following types of special non-K\" ahler Hermitian metrics on compact complex manifolds: {\it locally conformally K\" ahler}, {\it $k$-Gauduchon}, {\it balanced} and {\it locally conformally balanced} and prove that a locally conformally K\" ahler compact nilmanifold carrying a balanced or a left-invariant $k$-Gauduchon metric is necessarily a torus. Combined with the main result in \cite{fv16}, this leads to the fact that a compact complex 2-step nilmanifold endowed with whichever  two of the following types of metrics: balanced, pluriclosed and locally conformally K\" ahler is a torus. Moreover, we construct a family of compact nilmanifolds in any dimension carrying both balanced and locally conformally balanced metrics and finally we show a compact complex nilmanifold does not support a left-invariant locally conformally hyperK\" ahler structure.
		
		\hfill
		
		
	\end{abstract}
	
	\maketitle

\section{Introduction}

The first step towards understanding the non-Kähler world is to study classes of complex manifolds that admit interesting Hermitian metrics which are close to being Kähler from the cohomological and conformal point of view. One way of doing this is to impose several conditions on the fundamental form of the metric which generalize the notion of being closed. In this paper we shall consider the following special Hermitian metrics:

\begin{definition}\label{def1}
\begin{enumerate}

\item {\it Locally conformally Kähler}, widely studied since I. Vaisman's paper  \cite{vai76}. A Hermitian metric $g$ is called {\it locally conformally K\"ahler (lcK)} if there exists a covering $(U_i)_i$ of the manifold and smooth functions $f_i$ on each $U_i$ such that $e^{-f_i}g$ is K\" ahler. The definition is conformally invariant and is equivalent to the existence of a closed one-form $\theta$ (called the {\em Lee form} such that the fundamental form $\omega$ of the metric $g$ satisfies $d\omega=\theta \wedge \omega$. If $\theta$ is exact, the metric is called {\em globally conformally K\"ahler (gcK)}, otherwise we shall call it {\it strictly lcK}. If the Lee form is parallel w.r.t. the Levi-Civita connection of $g$, the lcK metric is called {\em Vaisman}.

\item {\it  Pluriclosed}, also referred to in the literature as {\em strongly K\"ahler with torsion}.  A Hermitian metric $g$ is called {\it pluriclosed} if its fundamental form $\omega$ satisfies $dd^c \omega=0$. See {\em e.g.}  \cite{bis89}.

\item {\it  k-Gauduchon}\footnote{Obviously, a special attention should be paid to $k=75$.}. A Hermitian metric $g$ is called {\it k-Gauduchon} if its fundamental form $\omega$ satisfies $dd^c\omega^{k} \wedge \omega^{n-k-1}=0$. They were first introduced in \cite{fww13} and they generalize the notion of {\it Gauduchon metric}, which corresponds precisely to $k=n-1$. The latter exists and is unique in any conformal class on a compact manifold, up to multiplication by positive constants, due to a celebrated result of Gauduchon. Among the special cases of $k$-Gauduchon metrics we have pluriclosed (1-Gauduchon) and {\it astheno-K\" ahler} (\ie $dd^c\omega^{n-2}$=0, hence, $(n-2)$-Gauduchon).

\item {\it Balanced}, (or {\em semi-K\"ahler}). A Hermitian metric $g$ is called {\it balanced} if its fundamental form $\omega$ satisfies $d\omega^{n-1}=0$, or equivalently if $\omega$ is co-closed. See  {\em e.g.} M.L. Michelson's paper \cite{m82}

\item {\it  Locally conformally balanced}. A Hermitian metric $g$ is called {\it locally conformally balanced} (lcb) if the Lee form $\theta$ of its fundamental form $\omega$ is closed. LcK metrics are in particular lcb. 

\end{enumerate}
\end{definition}

\medskip

The definitions above generalize the \K \ condition ($d \omega = 0$), and we shall be interested in these metrics precisely when this is not satisfied. For all these Hermitian metrics examples are abundant, especially on compact manifolds. In general, there are few (topological) restrictions for the existence of such metrics. An interesting problem then arises: can  these metrics  co-exist on the same manifold? This translates into the following questions:

\medskip

\begin{question}\label{q1} Given a complex manifold, is it possible that a Hermitian metric on it belong to two of the above classes? If the answer is positive, does such a metric have any additional property?
\end{question}

\medskip
	
\begin{question}\label{q2} Given a complex manifold, can it support two different Hermitian metrics belonging to different classes above?
\end{question}

\medskip

Only partial answers were given up to now and only for compact manifolds. In chronological order, the first result of this type concerns \ref{q2} and states the incompatibility between a K\" ahler metric and a strictly lcK metric on a compact complex manifold with fixed complex structure:

\begin{theorem}\textnormal{(\cite[Theorem 2.1]{vai80})} Let $(M, J)$ be a compact complex manifold admitting a K\" ahler metric. Then any lcK metric is gcK.
\end{theorem}

\medskip

A second result gave an answer for \ref{q1} and it addressed the compatibility in a given conformal class:

\begin{theorem}\textnormal{(\cite[Theorem 1.3, Proposition 3.8]{ip}, see also \cite{ai03})} Let $(M, J)$ be a compact complex manifold. Then for any Hermitian metric, the conditions lcK and $k$-Gauduchon, respectively balanced and $k$-Gauduchon are mutually incompatible in a given conformal class.
	
\end{theorem}

\medskip

Finally, the only answers for \ref{q2} we are aware of refer either to the class of nilmanifolds or are a generalization of Vaisman's result. 

\begin{theorem}\textnormal{(\cite[Theorem 1.1]{fv16}, \cite[Theorem A]{fv19})} If a compact nilmanifold $\Gamma \backslash G$ endowed with a left-invariant structure $J$ such that $[\mathfrak{g}, \mathfrak{g}]+ J[\mathfrak{g}, \mathfrak{g}]$ is abelian carries both a pluriclosed and a balanced metric, then it is necessarily a complex torus.
\end{theorem}

\begin{theorem}\textnormal{(\cite[Theorem 2.5]{au05})} On a compact complex manifold endowed with a \K \ metric $\omega$, any lcb metric with respect to the same complex structure is in fact gcb. 
\end{theorem}

\medskip

The aim of this paper is to give further partial answers to \ref{q1} and \ref{q2}, by checking the compatibility between two metrics of special type on a compact complex manifold within the class of compact nilmanifolds, endowed with a left-invariant complex structure, as they are a good ground for gaining more evidence towards formulating conjectures.

Section \ref{preli} presents the necessary background on nilmanifolds. In particular, it contains a short proof of the characterization of  left-invariant complex structures $J$ on 2-step nilmanifolds with $J$-invariant center. Section \ref{same_metric} discusses the $k$-Gauduchon condition  on compact manifolds. On the way, we obtain some formulae which  will be essential in the proof of the results on nilmanifolds. Section \ref{nila} is devoted to complex nilmanifolds and contains our main results. It starts with Subsection \ref{formalee} comprising results referring to the set of possible Lee forms on complex nilmanifolds, including a technical lemma giving the explicit formula for the Lee form of any Hermitian metric on 2-step complex nilmanifolds with left-invariant structures and complex invariant center. In Subsection \ref{compa} we give two constructions of complex nilmanifolds, the first one a family carrying an lcb metric which is also pluriclosed, the second one a family endowed with both balanced and strictly lcb metrics. Finally, in Subsection \ref{incompa} we show that  compact complex nilmanifolds endowed with (a) both a balanced and an lcK metric, or (b) both a left-invariant $k$-Gauduchon and an lcK metric, or (c) a pluriclosed and an lcK metric, are necessarily complex tori. We end by showing that a compact complex nilmanifold cannot be endowed with a left-invariant locally conformally hyperK\" ahler structure.

\medskip

\noindent{\bf Conventions:} Throughout the paper we shall use the conventions from \cite[(2.1)]{bes87} for the complex structure $J$ acting on complex forms on a manifold $(M, J)$. Namely:
\begin{itemize}
	\item $J\alpha=\mathrm{i}^{q-p} \alpha$, for any $\alpha \in \Lambda^{p, q}_{\mathbb{C}}M$, or equivalently, $J\eta(X_1, \ldots, X_p) = (-1)^p\eta(JX_1, \ldots, JX_p)$;
	\item the fundamental form of a Hermitian metric is given by $\omega (X, Y): = g(JX, Y)$;
	\item the operator $d^c$ is defined as $d^c := -J^{-1}dJ$, where $J^{-1} = (-1)^{\deg \alpha} J$. 
	\item The Lee form of a Hermitian metric $\omega$ is the trace of the torsion of the Chern connection associated to $\omega$, or equivalently given by $\theta=Jd^*\omega$, or by $d\omega^{n-1}=\theta \wedge \omega^{n-1}$ (see \cite[Eq. (15), (16)]{gau84}).
\end{itemize}

\section{Preliminaries}\label{preli}

This section is devoted to reviewing several basic facts that we shall use repeatedly in this paper.

\begin{definition} A {\it nilmanifold} $\Gamma \backslash G$ is a compact quotient of a simply connected nilpotent Lie group $G$ by a discrete subgroup $\Gamma$.

\end{definition}

\begin{definition} For a nilpotent Lie algebra $\mathfrak{g}$ the following descending series 
\begin{equation}
\mathcal{C}^0\mathfrak{g}:=\mathfrak{g}, \ \ \mathcal{C}^{i+1}\mathfrak{g}:=[\mathcal{C}^i\mathfrak{g}, \mathfrak{g}], \ \ i \geq 0
\end{equation}
is called the {\it derived series} and there exists $k\geq 1$ such that $\mathcal{C}^{l}\mathfrak{g}=0$, for any $l \geq k$. We call $k$ the {\it nilpotency step} of $\mathfrak{g}$ if $\mathcal{C}^{k-1}\mathfrak{g} \neq 0$.
\end{definition}

\medskip

\begin{definition}
	A differential object on $M$ which is induced by projection from a left-invariant differential object on $G$ is called {\em left-invariant}.	
\end{definition}

\begin{remark}\label{left_right}
We stress that the above is just a convenient name for the objects on $M$ induced form $G$: since $G$ has a natural right action on $M=\Gamma \backslash G$, but (in general) not a left one, these {\em left-invariant} objects on $M$ are are not invariant for the $G$-action.	
\end{remark}

\begin{definition} A complex structure $J$ on a Lie algebra $\mathfrak{g}$ is an endomorphism of $\mathfrak{g}$ with $J^2=-\mathrm{id}_{\mathfrak{g}}$, satisfying
\begin{equation}\label{integrabil}
J[X, Y]-[JX, Y]-[X, JY]-J[JX, JY]=0,
\end{equation}
for any $X, Y \in \mathfrak{g}$.
\end{definition}

\medskip

We note here that \eqref{integrabil} is equivalent to the fact that $\mathfrak{g}^{1,0}$, the $\mathrm{i}$-eigenspace of $J$ seen as as a space of endomorphism of $\mathfrak{g}^{\mathbb{C}}$, is a complex subalgebra. Defining $J$ on the tangent bundle $TG$ by means of left multiplications, we endow $G$ with a complex structure, left-invariant by construction. However, this does not make $G$ a complex Lie group, since $J$ is not necessarily right-invariant. If this was the case, $J$ would be a {\it bi-invariant} structure and $\mathfrak{g}$ would be a complex Lie algebra, hence the Lie bracket would be $\mathbb{C}$-linear, meaning 
\begin{equation}\label{c-liniar}
J[X, Y]=[JX, Y].
\end{equation}
Equation \eqref{c-liniar} clearly implies \eqref{integrabil}.

\medskip

Unlike complex structures, there is a bi-invariant object whose existence is always granted on compact nilmanifolds by the following result: 

\begin{theorem} {\rm(\cite[Lemma 6.2]{Mil})}
Any simply connected Lie group which admits a discrete subgroup with compact quotient is endowed with a bi-invariant volume form (that we denote $d\mathrm{vol}$).
\end{theorem}

\medskip

We shall use the bi-invariant volume for proving the existence of several left-invariant objects on nilmanifolds.

\medskip

We claim that a 2-step nilmanifold with $J$-invariant center has very particular structure equations, which we now describe for later use. For the sake of completeness we include a proof of the following result:
\begin{proposition}
	\label{prop:2step}
	A 2-step complex nilmanifold $(\Gamma \backslash G, J)$ with a left-invariant complex structure $J$ and $J$-invariant center has a $(1,0)$ co-frame $\{\alpha_1, ..., \alpha_n \}$ satisfying the structure equations
	\begin{equation}\label{structure2-step}
	\left\{
	\begin{array}{ll}
	d\alpha_i=0,  &1 \leq i \leq k, \\[.1in]
	d\alpha_i= \displaystyle\sum\limits_{r,s=1}^{k} \left( \frac{1}{2} c^i_{rs} \alpha_r \wedge \alpha_s + c^i_{r \overline{s}} \alpha_r \wedge \overline{\alpha}_s \right), &k < i \le n,
	\end{array}
	\right.
	\end{equation}
	for some $1 \le k < n - 1$ and constants $c^i_{rs}, c^i_{r \overline{s}} \in \mathbb{C}$. 
\end{proposition}
\begin{proof}
	Since $\Gamma \backslash G$ is 2-step, $[\mathfrak{g}, \mathfrak{g}]$ is included in $\mathfrak{c}:=\{X \in \mathfrak{g} \mid [X, \mathfrak{g}]=0\}$, the center of $\mathfrak{g}$. Therefore, $\mathfrak{t}:=\mathfrak{g}/\mathfrak{c}$ is an abelian Lie algebra and moreover is endowed with a complex structure, since, by hypothesis, $J$ descends to the quotient. Consequently, as vector spaces only, we have the splitting $\mathfrak{g}^{1,0}=\mathfrak{t}^{1,0} \oplus \mathfrak{c}^{1,0}.$
	
	Let now $0 \neq k := \dim_{\mathbb{C}} \mathfrak{t}^{1,0}$ and consider $\{X_1, \ldots, X_n\}$ a basis for $\mathfrak{g}^{1,0}$ such that $\{X_1, \ldots, X_k\}$ is a basis for $\mathfrak{t}^{1,0}$ and $\{X_{k+1}, \ldots, X_n\}$ a basis for $\mathfrak{c}^{1,0}$. Since $\mathfrak{t}^{1,0}$ is an abelian Lie algebra, we have for any $1 \leq r, s \leq k$:
	\begin{equation}\label{ecstruct}
	\begin{split}
	[X_r, X_s] &= -\sum_{i=k+1}^n \tfrac{1}{2}c^i_{rs} X_i\\
	[X_r, \overline{X}_s] &= - \sum_{k+1}^n c^i_{r\overline{s}}X_i + c^i_{s\overline{r}}\overline{X}_i.
	\end{split}
	\end{equation}
	Taking now the dual co-frame $\{\alpha_1, \ldots, \alpha_n\}$ of $\{X_1, \ldots, X_n\}$, we see that \eqref{ecstruct} implies  \eqref{structure2-step}. In particular, this shows that $\Gamma \backslash G$ is a holomorphic principal torus bundle over a real torus of dimension $\mathrm{dim}_{\mathbb{R}}\mathfrak{g}-\mathrm{dim}_{\mathbb{R}}\mathfrak{c}$. 
\end{proof}

\smallskip

\begin{remark}\label{defechivnil}
	Note that, in the above case, the co-frame $\{\alpha_1, \ldots, \alpha_n\}$ of $\mathfrak{g}^{1,0}$ has the property that $d\alpha_i$ belongs to $\Lambda^2\langle \alpha_1, \ldots,  \alpha_{i-1},\overline{\alpha}_1, \ldots,  \overline{\alpha}_{i-1} \rangle$, the ideal generated in $\mathfrak{g}^{\C}$ by $\{\alpha_1, \ldots,  \alpha_{i-1}\}$, for $1 \leq i \leq \mathrm{dim}_{\mathbb{C}} \,\mathfrak{g}^{1,0}$. This is equivalent to $J$ being \textit{nilpotent} (see \cite{cfg86}, \cite[Theorem 2]{cfgu97}), where this is defined as:
\end{remark}

\begin{definition}\label{nilpotentJ} A complex structure $J$ on a nilpotent Lie algebra $\mathfrak{g}$ is called {\it nilpotent} if the ascending series $\{\mathfrak{g}_i^J\}_{i \geq 0}$ given by
\begin{equation}
\mathfrak{g}_0^J=0, \ \ \ \ \ 
\mathfrak{g}_i^J=\{X \in \mathfrak{g} \mid [J^kX, \mathfrak{g}] \subseteq \mathfrak{g}_{i-1}^J,\ \ k \in \{0, 1\} \}
\end{equation}
has the property that $\mathfrak{g}_l^J = \mathfrak{g}$, for some $l>0$.
\end{definition}

\medskip

\begin{remark}\label{bi-invabelian} When the complex structure $J$ is bi-invariant, then $c^i_{r\overline{s}} =0$, for any $1 \leq r, s \leq k$ and $k+1 \leq i \leq n$. Indeed, by using \eqref{c-liniar}, we can easily see by direct computation that $d\alpha_i(X-\mathrm{i}JX, Y+\mathrm{i}JY)=0$, for any real left-invariant vector fields $X$ and $Y$. As a matter of fact, this argument shows that in general, on a nilmanifold endowed with a bi-invariant structure,  $d(\Lambda^{1,0}) \subseteq \Lambda^{2,0}$. Moreover, in this case, $J$ is automatically nilpotent by \ref{nilpotentJ} and \eqref{c-liniar}, and there exists a co-frame $\{\alpha_1, \ldots, \alpha_n\}$ such that $d\alpha_i \in \langle \alpha_1, \ldots, \alpha_{i-1} \rangle$ (see \ref{defechivnil}).  If the complex structure $J$ is {\it abelian}, meaning that $[JX, JY]=[X, Y]$, for any $X, Y \in \mathfrak{g}$, then again by straightforward computation, we obtain $c^{i}_{rs}=0$. We conclude that $J$ cannot be abelian and bi-invariant at the same time, unless $\mathfrak{g}$ is a complex abelian Lie algebra.
\end{remark}

\section{$k$-Gauduchon metrics}\label{same_metric}

We begin by noticing that in the case of complex surfaces, the notion of $k$-Gauduchon metric makes sense only for $k=1$ and in this case, it is precisely a Gauduchon metric, coinciding also with the definion of pluriclosed. Therefore they always exist and are unique up to scalar multiplication in a given conformal class. 

\begin{proposition}
\label{prop:aceeasiForma}
Let $(M, J)$ be a compact complex manifold of complex dimension $n>2$ and $\omega$ a Hermitian metric. Then:

\begin{enumerate}
	\item[\rm{i)}] $\omega$ is $k$-Gauduchon for any $1\leq k \leq n-1$ if and only if
	\begin{equation}
	(n-k-1)dJd\omega \wedge \omega^{n-2}=-(k-1)d\star \theta,
	\end{equation}
	where $\star$ is the Hodge operator.
	\item[\rm{ii)}] If $\omega$ is $k_1$-Gauduchon and $k_2$-Gauduchon for some $1 \leq k_1 \neq k_2 \leq n-1$, then $\omega$ is $k$-Gauduchon, for any $1\leq k \leq n-1$ (in particular, it is Gauduchon).
\end{enumerate}

\begin{proof}
	
	\rm{i)} The metric $\omega$ is $k$-Gauduchon if and only if:
	\begin{equation}
	\label{eq:aceeasiForma1}
	\begin{split}
	0 &= dd^c \omega^k \wedge \omega^{(n-k-1)} = dJd (\omega^k) \wedge \omega^{n-k-1} \\ 
	&= kdJ(d\omega \wedge \omega^{k-1}) \wedge \omega^{n-k-1} = kd(Jd\omega \wedge \omega^{k-1}) \wedge \omega^{n-k-1} \\
	&= kdJd\omega \wedge \omega^{k-1} \wedge \omega^{n-k-1} - k(k-1)Jd\omega \wedge d\omega \wedge \omega^{k-2} \wedge  \omega^{n-k-1} \\
	&= kdJd\omega \wedge \omega^{n-2} - k(k-1)Jd\omega \wedge d\omega \wedge \omega^{n-3}, 
	\end{split}
	\end{equation}
	
	thus
	
	\begin{equation}
	\label{eq:k-Gaud}
	dJd\omega \wedge \omega^{n-2} = \frac{k-1}{n-2}Jd\omega \wedge d(\omega^{n-2}), 
	\end{equation}
which gives:
\begin{equation*}
dJd\omega \wedge \omega^{n-2} = \frac{k-1}{n-2}\left( dJd\omega \wedge \omega^{n-2}-d(Jd\omega \wedge \omega^{n-2})\right),
\end{equation*}
and furthermore,
\begin{equation*}
(n-k-1) dJd\omega \wedge \omega^{n-1}=-(k-1)dJd\omega^{n-1}.
\end{equation*}
Now, since $d\omega^{n-1}=\theta \wedge \omega^{n-1}$ and $\star \theta = \frac{J\theta \wedge \Omega^{n-1}}{(n-1)!}$ (see \cite[Eq. (46)]{gau84}), we conclude this is equivalent to
\begin{equation*}
	(n-k-1)dJd\omega \wedge \omega^{n-2}=-(k-1)d\star \theta.
	\end{equation*}	

In particular, if $\omega$ is $k$-Gauduchon with $k \neq n-1$, then:	
	\begin{equation}
	\label{eq:1G-integrala}
	\int_M dd^c\omega \wedge \omega^{n-2} = 0,
	\end{equation}
	as $k \neq n-1$.

	ii) If $\omega$ is $k_1$ and $k_2$-Gauduchon, from \eqref{eq:k-Gaud} we get
	\[
	dJd\omega \wedge \omega^{n-2} = \frac{k_1-1}{n-2}Jd\omega \wedge d(\omega^{n-2}) = \frac{k_2-1}{n-2}Jd\omega \wedge d(\omega^{n-2}).
	\]
	Note that \eqref{eq:k-Gaud} holds also for $k=n-1$. Now, since $k_1 \neq k_2$,
	\[
	dJd\omega \wedge \omega^{n-2} = Jd\omega \wedge d(\omega^{n-2}) = 0,
	\]
	and, for any $1 \le k < n-1$, using \eqref{eq:aceeasiForma1},
	\[
	0 = kdJd\omega \wedge \omega^{n-2} - k(k-1)Jd\omega \wedge d\omega \wedge \omega^{n-3} = dd^c \omega^k \wedge \omega^{n-k-1}
	\]
	\ie $\omega$ is $k$-Gauduchon. Moreover, $dJd\omega^{n-1}=dJd\omega \wedge \omega^{n-2}-(n-2)Jd\omega \wedge \omega^{n-2}=0$, hence $\omega$ is also Gauduchon.
\end{proof}

\end{proposition}

\section{Non-K\" ahler structures on nilmanifolds}\label{nila}

\subsection{Lee forms on nilmanifolds}\label{formalee}

\hfill\\

Many of the special properties of a Hermitian metric are encoded in its Lee form. It is an open problem in general to determine the subspace of 1-forms which can be Lee forms for some Hermitian metric on a compact complex manifold. In what follows we give a technical lemma computing any possible Lee form on a special type of $2$-step nilmanifolds and continue with some general restrictions on any kind of nilmanifolds.	

The following result of Michelsohn (see \cite[Lemma 4.8]{m82}) will be useful:

\begin{proposition}\label{n-1root} Let $\tilde{\omega}$ be an $(n-1, n-1)$ positive form, then there exists a unique positive $(1, 1)$-form $\omega$ such that $\omega^{n-1}=\tilde{\omega}$.
\end{proposition}

\medskip

In the light of this result, we shall produce balanced (or lcb) metrics by constructing closed $(n-1, n-1)$ positive forms $\tilde{\omega}$ (or satisfying $d\tilde{\omega}=\theta \wedge \tilde{\omega}$ for a closed one-form $\theta$), rather than considering the metric itself. Note that if $\tilde{\omega}$ is invariant, $\omega$ also turns out to be invariant by the proof of \ref{n-1root}. In general, we shall represent any real $(n-1, n-1)$ form as 
\begin{equation*}
\tilde{\omega} = \mathrm{i}^{n-1} \sum\limits_{i,j=1}^n a_{i \overline{j}} m_{i \overline{j}} 
\end{equation*}
where
\begin{equation*}
m_{i \overline{j}} = \alpha_1 \wedge \overline{\alpha}_1 \wedge ... \wedge \widehat{\alpha_i} \wedge \overline{\alpha}_i \wedge ... \wedge  \alpha_j \wedge \widehat{\overline{\alpha}_j} \wedge ... \wedge \alpha_n \wedge \overline{\alpha}_n,
\end{equation*}
$a_{i \overline{j}} = - \overline{a_{j \overline{i}}}$, for any $1 \le i\neq j \le n$ and $a_{i \overline{i}} = \overline{a_{i \overline{i}}}$, for any $1 \le i \le n$. The positivity of $\tilde{\omega}$ is equivalent to the positive definiteness of the matrix $\tilde{A}$, where 
\[
\tilde{A}_{i j} = 
\left\{
\begin{array}{ll}
\hspace{0.3cm} a_{i \overline{j}},  &\text{ if } i \le j, \\[.1in]
-a_{i \overline{j}},  &\text{ if } i > j.
\end{array}
\right.
\]
Indeed, this follows from the fact that with respect to the standard metric $g_0:= \mathrm{i} \sum^{n}_{i=1}\alpha_i \wedge \overline{\alpha}_i$, we have
\[
*_{g_0}\tilde{\omega} = \sum\limits_{i,j=1}^n \tilde{A}_{i j} \alpha_i \wedge \overline{\alpha}_j,
\] 
and this has to be a positive $(1, 1)$-form (see \cite[Chapter III]{d12}). In general, an $(n-1, n-1)$-form $\eta$ is positive if and only if $*_g\eta$ is positive for any Hermitian metric $g$.

\bigskip

The following lemma gives an explicit formula for the Lee form of any Hermitian metric on a $2$-step complex nilmanifold endowed with a left-invariant complex structure with $J$-invariant center.

\begin{lemma}
	\label{lem:2step+lcb+b}
	Let $(M = \Gamma \backslash G, J)$ be a 2-step complex nilmanifold with left-invariant $J$ and $J$-invariant center and $\tilde{\omega}=\mathrm{i}^{n-1} \sum\limits_{i,j=1}^n a_{i \overline{j}} m_{i \overline{j}}$ an $(n-1, n-1)$ positive form on $M$, as described above. The Lee form of the metric $\omega$ with $\omega^{n-1}=\tilde{\omega}$ is given by
	\begin{equation}
	\theta = \sum_{i=1}^n a_i (\alpha_i + \overline{\alpha}_i), 
	\end{equation}
	with $a_i$ uniquely determined by:
	\begin{equation}\label{ec:lcb}
	\tilde{A} \cdot (a_1, ..., a_n)^t = b,
	\end{equation}
	where $b \in \mathbb{C}^n$, 
	\[
	b_l = 
	\left\{
	\begin{array}{ll}
	\displaystyle\sum_{i,j=1}^k \tilde{A}_{i j} c_{i \overline{j}}^l,  &\text{\rm if }\  l > k, \\[.2in]
	0,  &\text{\rm if }\  l \le k.
	\end{array}
	\right.
	\]
	Moreover, $\theta$ is closed (hence, $\omega$ is lcb) when $\sum\limits_{i=k+1}^n a_i c^i_{rs} = 0$ and $\sum\limits_{i=k+1}^n a_i (c^i_{r \overline{s}} - c^{\overline{i}}_{s \overline{r}}) = 0$ for all $1 \le r, s \le k$.
	In particular, $\omega$ is balanced if and only if $b = 0$.

\end{lemma}

\begin{proof}
	This is shown by straightforward computation:
	
	For the first part of the statement, as the Lee form is unique, we simply have to verify the equality $d\tilde{\omega}=\theta \wedge \tilde{\omega}^{n-1}$. Firstly, using the description of the structure equations given by \ref{prop:2step},
	\[
	dm_{i \overline{j}} = 
	\left\{
	\begin{array}{ll}
		\displaystyle\sum_{l = k + 1}^n \left(c_{i \overline{j}}^l m_l + \overline{c_{j \overline{i}}^l} m_{\overline{l}}\right),  &\text{\rm if }\  1 \le i \le j \le k, \\[.2in]
		\displaystyle -\sum_{l = k + 1}^n \left(c_{i \overline{j}}^l m_l + \overline{c_{j \overline{i}}^l} m_{\overline{l}}\right),  &\text{\rm if }\  1 \le j < i \le k, \\[.2in]
		0,  &\text{\rm otherwise, }\
	\end{array}
	\right.
	\]
	where $m_l = \alpha_1 \wedge \overline{\alpha}_1 \wedge ... \wedge \widehat{\alpha_l} \wedge \overline{\alpha}_l \wedge ... \wedge \alpha_n \wedge \overline{\alpha}_n$ and $m_{\overline{l}} = \alpha_1 \wedge \overline{\alpha}_1 \wedge ... \wedge \alpha_l \wedge \widehat{\overline{\alpha}_l} \wedge ... \wedge \alpha_n \wedge \overline{\alpha}_n$. Consequently, using the definition of the matrix $\tilde{A}$,
	\begin{equation}
		\label{4.2eq1}
		\begin{split}
			d \tilde{\omega} &= i^{n-1} \sum_{i, j=1}^k a_{i \overline{j}} \cdot 
			\left\{
			\begin{array}{ll}
				\displaystyle\sum_{l = k + 1}^n \left(c_{i \overline{j}}^l m_l + \overline{c_{j \overline{i}}^l} m_{\overline{l}}\right),  &\text{\rm if }\  1 \le i \le j \le k, \\[.2in]
				\displaystyle -\sum_{l = k + 1}^n \left(c_{i \overline{j}}^l m_l + \overline{c_{j \overline{i}}^l} m_{\overline{l}}\right),  &\text{\rm if }\  1 \le j < i \le k, 
			\end{array}
			\right. 
			\\ &= i^{n-1} \sum_{i, j=1}^k \tilde{A}_{i j} \sum_{l = k + 1}^n \left(c_{i \overline{j}}^l m_l + \overline{c_{j \overline{i}}^l} m_{\overline{l}}\right).
		\end{split}
	\end{equation}

	On the other hand, 
	\begin{equation}
		\label{4.2eq2}
		\begin{split}
			\theta \wedge \tilde{\omega} &= i^{n-1} \cdot
			\left\{
			\begin{array}{ll}
				\displaystyle \sum_{i, j = 1}^n \left( a_i a_{i \overline{j}} m_{\overline{j}} + a_j a_{i \overline{j}}m_i \right),  &\text{\rm if }\  1 \le i \le j \le n, \\[.2in]
				\displaystyle -\sum_{i, j = 1}^n \left( a_i a_{i \overline{j}} m_{\overline{j}} + a_j a_{i \overline{j}}m_i \right),  &\text{\rm if }\  1 \le j < i \le n, 
			\end{array}
			\right. 
			\\ &= i^{n-1} \sum_{i, j = 1}^n \left( a_i \tilde{A}_{i j} m_{\overline{j}} + a_j \tilde{A}_{i j}m_i \right) \\
			&= i^{n-1} \sum_{i, l = 1}^n \left( a_i \tilde{A}_{i l} m_{\overline{l}} + a_i \tilde{A}_{l i}m_l \right).
		\end{split}
	\end{equation}
	Equaling \eqref{4.2eq1} and \eqref{4.2eq2}, we get
	\[
	\left\{
	\begin{array}{ll}
		\displaystyle \sum_{i = 1}^n \tilde{A}_{li} a_i = \sum_{i,j=1}^k \tilde{A}_{ij} c_{i \overline{j}}^l, \ \forall \ l > k, \\[.2in]
		\displaystyle \sum_{i = 1}^n \tilde{A}_{li} a_i = 0, \ \forall \ l \le k, \\[.2in]
		\displaystyle \sum_{i = 1}^n \tilde{A}_{il} a_i = \sum_{i,j=1}^k \tilde{A}_{ij} \overline{c_{j \overline{i}}^l}, \ \forall \ l > k, \\[.2in]
		\displaystyle \sum_{i = 1}^n \tilde{A}_{il} a_i = 0, \ \forall \ l \le k.
	\end{array}
	\right. 
	\]
	However, as $\overline{\tilde{A}_{il}} = \tilde{A}_{li}$ and $a_i \in \R$, the last two equations are equivalent to the first two, so we are left with the conditions stated.
	
	We end by computing $d\theta$, again using \ref{prop:2step}:
	\[
	d\theta = \sum_{i = k+1}^n a_i \sum_{r, s = 1}^k \left( \frac{1}{2} c_{rs}^i \alpha_r \wedge \alpha_s + \left( c_{r \overline{s}}^i - \overline{c_{s \overline{r}}^i} \right) \alpha_r \wedge \overline{\alpha}_s + \frac{1}{2} \overline{c_{rs}^i} \overline{\alpha}_r \wedge \overline{\alpha}_s \right),
	\]
	so indeed $d \theta = 0$ if and only if $\sum\limits_{i=k+1}^n a_i c^i_{rs} = 0$ and $\sum\limits_{i=k+1}^n a_i (c^i_{r \overline{s}} - c^{\overline{i}}_{s \overline{r}}) = 0$ for all $1 \le r, s \le k$.
\end{proof}  

\medskip

Let us define the following subspace of left-invariant $1$-forms:
$$\mathfrak{lee}:=\{\theta \in \mathfrak{g}^* \mid \exists \ \omega \text{ left-invariant metric}, \theta=Jd^*\omega\}.$$

The following results give some restrictions on the subspace $\mathfrak{lee}$.
\begin{proposition}
Let $(\Gamma \backslash G, J)$ be a complex compact nilmanifold with left-invariant $J$. Denote by $\mathfrak{h}$ the subspace of $\mathfrak{g}^{1,0}$ generated by closed forms. Then $\mathfrak{h} \cap \mathfrak{lee}$ is at most $\{0\}$.
\end{proposition}

\begin{proof} 
Let $\mathfrak{h}$ be generated by $\alpha_1, \ldots, \alpha_k$.  In the case there exists $\theta=\sum^k_{i=1} (r_i \alpha_i+\overline{r}_i\overline{\alpha}_i) \in \mathfrak{h} \cap \mathfrak{lee}$, then notice that since $d\alpha_i=0$, for any $1 \leq i \leq k$, $dJ\theta=0$. If $\omega$ was a Hermitian metric such that $\theta=Jd^*\omega$, then $d\omega^{n-1}=\theta \wedge \omega^{n-1}$, then, by Stokes and \cite[Rel. (46)]{gau84}:
\begin{equation}
0=\int_M dJ\theta \wedge \omega^{n-1}=\int_M J\theta \wedge \theta \wedge \omega^{n-1}=-(n-1)!\int_M||\theta||^2 d\rm{vol}_g,
\end{equation}
hence $\theta=0$.  
\end{proof} 

\begin{proposition} Let ($\Gamma \backslash G, J$) be a complex compact nilmanifold with left-invariant nilpotent $J$. The following are equivalent:
\begin{enumerate}
\item[$(1)$] any $(n-1, n-1)$-form is closed;
\item[$(2)$] $d\Lambda^{1,0} \subseteq \Lambda^{2,0}$;
\item[$(3)$] $J$ is bi-invariant.
\end{enumerate}
In particular, in any of the above equivalent situations, $\mathfrak{lee}=0$, therefore any Hermitian metric is balanced.
\end{proposition}
\begin{proof}
Since $J$ is nilpotent, there exists a co-frame of $\mathfrak{g}^{1, 0}$ such that 
\begin{equation}
d\alpha_l=\sum_{i, j=1}^{l-1} \frac{1}{2}c^l_{ij}\alpha_i \wedge \alpha_{j} + c^l_{i\overline{j}}\alpha_i \wedge \overline{\alpha}_j, 1\leq l \leq n.
\end{equation}
$(1) \Leftrightarrow (2)$ This is immediate by the following equalities:
\begin{equation}
dm_{i\overline{j}}=\sum_{l>j}^nc_{i\overline{j}}^{l}m_{l}+\overline{c}_{j\overline{i}}^{l}m_{\overline{l}}, 1\leq i < j \leq n, 
\end{equation}
\begin{equation}
dm_{i\overline{i}}=\sum_{l>i}^nc_{i\overline{i}}^{l}m_{l}, 1 \leq i \leq n.
\end{equation}
Therefore, any $c^l_{i\overline{j}}=0$, $i \neq j$, $l>\mathrm{max}(i, j)$ is equivalent to $d\Lambda^{1,0} \subseteq \Lambda^{2,0}$ and to the closedness of any $(n-1, n-1)$-form.

$(2) \Leftrightarrow (3)$ This is clear by equation \eqref{c-liniar} and statement (2) being equivalent to $\alpha(X-\mathrm{i}JX, Y+\mathrm{i}JY)$, for any left-invariant $X$ and $Y$ and any $(1,0)$-form $\alpha$.
\end{proof}

\subsection{Compatibility results }\label{compa}

\hfill\\

Contrary to the expectation that, generally, two special conditions imposed to the same Hermitian metric imply K\" ahlerianity, we show this is not the case. We can give now the announced examples.

\subsubsection{Pluriclosed and locally conformally balanced} We first give an example in any complex dimension of a family of 2-step complex nilmanifolds endowed with a metric satisfying simultaneously the pluriclosed and the locally conformally balanced condition.

\begin{example} Let $q \in \mathbb{Q}^*$ and define $(\mathfrak{g}, J)$ the real Lie algebra endowed with a complex structure $J$ given by the $(1,0)$ co-frame $\{\alpha_1, \ldots, \alpha_n\}$ which satisfies the structure equations:
\begin{equation}\label{structureEx}
\left\{
    \begin{array}{ll}
       d\alpha_k=0, \ \ \ \  1 \leq k \leq n-1, \\[.1in]
       d\alpha_n= q \alpha_1 \wedge \overline{\alpha}_1.
    \end{array}
\right.
\end{equation}

By a classical result of Cartan (see \eg \cite[Page 152]{s65}), there exists a simply connected Lie group $G$ with Lie algebra $\mathfrak{g}$ on which we extend $J$ to a left-invariant complex structure. Using now a theorem of Malcev (\cite{m62}), since the structure equations of $G$ are rational, there exists  a discrete subgroup $\Gamma$ such that $\Gamma\backslash G$ is compact. Therefore, $\mathfrak{g}$ is the Lie algebra of a compact complex nilmanifold.

We claim that
\begin{equation*}
\omega := \sum^n_{j=1} \mathrm{i} \alpha_j \wedge \overline{\alpha}_j
\end{equation*}
is both pluriclosed and lcb on $\Gamma\backslash G$. Indeed, using \eqref{structureEx}, we get by straightforward computation that $dd^c \omega = \mathrm{i}dd^c (\alpha_n \wedge \overline{\alpha}_n)=0$. Moreover, by \eqref{ec:lcb} its Lee form is $\theta=q(\alpha_n+\overline{\alpha}_n)$, which is closed, hence $\omega$ is also lcb. \hfill{\blacksquare}
\end{example}

\medskip

\subsubsection{Balanced and locally conformally balanced} We next present an example of a compact complex nilmanifold carrying both a balanced and a strictly lcb metric, thus disproving the natural analogue of Vaisman's theorem for the balanced setting. Before giving the construction, we need the following result:

\begin{proposition}\label{notconf} Let $M = \Gamma \backslash G$ be a compact complex nilmanifold endowed with a left-invariant complex structure $J$. If it admits two left-invariant lcb metrics $\omega_1$ and $\omega_2$ such that their Lee forms $\theta_1$ and $\theta_2$ are different, then the metrics are not conformal.
\end{proposition}

\begin{proof} We notice first that $\theta_1$ and $\theta_2$ are left-invariant. Indeed, this follows form the fact that the Lee form $\theta$ of metric $\omega$ is uniquely determined by:
\begin{equation}\label{Lee}
\theta= Jd^* \omega,
\end{equation}
where $d^*$ is the adjoint of $d$ with respect to $\omega$, therefore the left-invariance is inherited from $\omega$. If $\theta_1-\theta_2=df$, then $df$ would be a left-invariant one-form. However, for any left-invariant field $X$, we have
\[
X(f) = df(X) = \theta(X) \in \mathbb{R},
\]
hence
\[
\int_M X(f) \dvol = \int_M \mathcal{L}_X (f \dvol) = \int_M di_X (f\dvol) = 0,
\]
where $\dvol$ is a bi-invariant volume form. Therefore, $X(f) = 0$ for all left-invariant vector fields, \ie $f$ is constant, which is false.
\end{proof}

\begin{remark} Note that \ref{notconf} holds in general for any compact quotient $\Gamma \backslash G$ endowed with a bi-invariant volume form.
	\end{remark}

\begin{example} 
As before, we consider a 2-step nilmanifold $M = \Gamma \backslash G$ of complex dimension $n \geq 3$, whose Lie algebra $\mathfrak{g}$ is given by the following structure equations:

\begin{equation*}
\left\{
\begin{array}{ll}
d\alpha_i &= 0, \ 1 \leq i \leq n-1, \\[.1in]
d\alpha_n &= \displaystyle\sum\limits_{r=1}^{n-1}  c^n_{r \overline{r}} \alpha_r \wedge \overline{\alpha}_r,
\end{array}
\right.
\end{equation*}
where $c^{n}_{r \overline{r}} \in \mathbb{Q}$, $c^n_{r \overline{r}} >0$, for all $1 \leq r \leq n-2$ and $c^n_{n-1 \overline{n-1}}<0$ such that $\sum^{n-1}_{r=1} c^n_{r \overline{r}} \neq 0$. Then 
$$\omega_1:= \mathrm{i}\sum_{i=1}^{n} \alpha_i \wedge \overline{\alpha}_i$$ 
is an lcb metric by \ref{lem:2step+lcb+b}. Its Lee form is
$$\theta=(\sum^{n-1}_{r=1} c^n_{r \overline{r}}) (\alpha_n + \overline{\alpha}_n),$$
 which is not exact by \ref{notconf}.

\medskip

To construct a balanced metric, choose $(a_{i \overline{j}})_{1 \le i, j \le n}$ such that

\begin{equation*}
\left\{
\begin{array}{ll}
a_{i \overline{j}} = 0, &\text{for all}\ \  1 \le i \neq j \le n, \\[.1in]
a_{i \overline{i}} > 0, &\text{for all}\ \  1 \le i \le n, \\[.1in]
a_{n-1 \overline{n-1}} = \displaystyle-\frac{\sum_{r=1}^{n-2} a_{r \overline{r}} c_{r \overline{r}}^n }{c^n_{n-1 \overline{n-1}}}.
\end{array}
\right.
\end{equation*}

One can check that equations \eqref{ec:lcb} are satisfied, the $(n-1, n-1)$-form $$\tilde{\omega}:=\mathrm{i}^{n-1} \sum^{n}_{i=1} a_{i \overline{i}}m_{i\overline{i}}$$ is positive, hence, by Lemma \ref{lem:2step+lcb+b}, it gives a balanced metric on $M$. Note that $\omega_1$ constructed above is not an lcK metric, by a result we shall prove in the sequel (see \ref{thm:lckbal}).

\end{example}

\hfill

\subsection{Incompatibility results}
\label{incompa}

\subsubsection{Balanced versus lcK}~

\begin{theorem}
\label{thm:lckbal}
Let $ M$ be a compact complex nilmanifold endowed with a left-invariant complex structure $J$. If it carries an lcK metric $\omega_1$ and a balanced metric $\omega_2$, both of them compatible with the complex structure $J$, then $(M, J)$ is a complex torus.

\end{theorem}

\begin{proof}
Assume $M$ is not a complex torus. Following the averaging procedure presented in \cite[Theorem 7]{bel00}, it was shown in \cite[Theorem 2.2]{fg04} that there also exists a left-invariant balanced metric $\omega_2$. Then, using the structure theorem of \cite{s07}, $(M, J)$ is biholomorphic to a quotient of $(H(n-1) \times \mathbb{R}, J_0)$, where $H(n-1)$ is the real Heisenberg group, thus its Lie algebra $\mathfrak{h}_{n-1} \times \mathbb{R}$ is generated by 
\begin{equation}\label{eq:Liebrackets}
\{X_i, Y_i, Z, A\} \text{ with } [X_i, Y_i] = Z, JX_i = Y_i \text{ and } JA=Z, \ 1 \le i \le n-1,
\end{equation}
therefore
\begin{equation}
\label{eq:lckbal1}
dx_i = dy_i = d\alpha = 0 \text{ and } dz = -\sum_{i=1}^{n-1} x_i \wedge y_i,
\end{equation}
where $\{x_i, y_i, z, \alpha \}$ is the corresponding co-frame.

Since $\omega_2$ is $(1,1)$ and positive, $\omega_2^{n-1}$ is a positive $(n-1, n-1)$ form (see \cite[p. 131]{d12}),
\[
\omega_2^{n-1}(X_1, Y_1, ..., \widehat{X_i, Y_i}, ..., X_{n-1}, Y_{n-1}, A, Z) = C_i > 0.
\] 
Consequently, using \eqref{eq:lckbal1}, all terms in $\omega_2^{n-1}$ apart from those of type
\[
C_i x_1 \wedge y_1 \wedge ... \wedge \widehat{x_i \wedge y_i} \wedge ... \wedge x_{n-1} \wedge y_{n-1} \wedge \alpha \wedge z
\]
are closed, thus
\[
d\omega_2^{n-1} = -(\sum_{i=1}^{n-1} C_i) x_1 \wedge y_1 \wedge ... \wedge x_{n-1} \wedge y_{n-1} \wedge \alpha \neq 0,
\]
a contradiction.
\end{proof}

\medskip

\begin{remark}
	Note that the assumption of left-invariance for the complex structure was crucial for the use of the averaging trick in \cite[Theorem 2.2]{fg04}.
\end{remark}

\subsubsection{$k$-Gauduchon versus lcK}

We need the following characterization of $k$-Gauduchon metrics:

\begin{lemma} {\rm(cf. \cite[Proposition 2.2]{lu17})}
\label{lem:k=1}
	Let $M$ be a compact complex nilmanifold of complex dimension $n$ with a left-invariant complex structure $J$ and  a left-invariant Hermitian metric $\omega$. Then $\omega$ is $k$-Gauduchon for some $1 \le k < n - 1$ if and only if it is $1$-Gauduchon. 
\end{lemma}

\medskip

\noindent{\em Proof.} 
	We use the fact that, on a nilmanifold, a left-invariant $2n$-form $\eta$ is a multiple of the volume form, so
	\[
	\eta = 0 \iff \int_M \eta = 0.
	\]
	
	With the above equivalence and from \eqref{eq:aceeasiForma1}, \eqref{eq:k-Gaud} and \eqref{eq:1G-integrala}, we have
	\[
	\omega \text{ is } k-\text{Gauduchon } \iff \int_M dd^c\omega \wedge \omega^{n-2} = 0, \iff \omega \text{ is } 1-\text{Gauduchon}. \qquad\qquad\qquad\qquad\quad\hfill{\blacksquare}
	\]

\begin{theorem}
	\label{thm:lckkG}
	Let $ M$ be a compact complex nilmanifold endowed with a left-invariant complex structure $J$. If it carries an lcK metric $\omega_1$ and a left invariant $k$-Gauduchon metric $\omega_2$, for some $1 \le k < n-1$, both of them compatible with the complex structure $J$, then $(M, J)$ is a complex torus.
	
\end{theorem}

\begin{proof}
	Assume $M$ is not a complex torus. Then, using the structure theorem of \cite{s07} as in \ref{thm:lckbal}, $(M, J)$ has the same Lie algebra structure.
	
	Since $\omega_2$ is of type $(1,1)$,
	\begin{equation*}
	\begin{split}
	\omega_2 &= \sum_{i,j=1}^{n-1} A_{i,j}(x_i \wedge x_j + y_i \wedge y_j) + \sum_{i,j=1}^{n-1} B_{i,j}(x_i \wedge y_j - y_i \wedge x_j) \\
	&+ \sum_{i=1}^{n-1} C_i(x_i \wedge \alpha + y_i \wedge z) + \sum_{i=1}^{n-1} D_i(x_i \wedge z - y_i \wedge \alpha) \\
	&+ E \alpha \wedge z.
	\end{split}
	\end{equation*}
	Notice that, since $\omega_2$ is positive, $E = \omega_2(A, Z) > 0$ and, as before, since $\omega_2^{n-2}$ is a positive $(n-2, n-2)$ form,
	\[
	\omega_2^{n-2}(X_1, Y_1, ..., \widehat{X_i, Y_i}, ..., ..., \widehat{X_j, Y_j}, ..., X_{n-1}, Y_{n-1}, A, Z) = C_{i,j} > 0.
	\]
	
	We first compute 
	\begin{equation}
	\label{eq:lckkG1}
	\begin{split}
	dd^c \omega_2 \wedge \omega_2^{n-2} &= dJd \omega_2 \wedge \omega_2^{n-2} \\
	&= E\sum_{i,j=1}^{n-1} x_i \wedge y_i \wedge x_j \wedge y_j \wedge \omega_2^{n-2} \\
	&= E(\sum_{i,j=1}^{n-1} C_{i,j}) \dvol. 
	\end{split}
	\end{equation}
	
	Now, since $\omega_2$ is $k$-Gauduchon, by \ref{lem:k=1}, it is $1$-Gauduchon, therefore
	\[
	0 = dd^c\omega_2 \wedge \omega_2^{n-2} = E\big(\sum_{i,j=1}^{n-1} C_{i,j}\big) \dvol \neq 0,
	\]
	a contradiction. 
\end{proof}

\begin{corollary}
\label{cor:SKTsilcK}
	Let $ M$ be a compact complex nilmanifold endowed with a left-invariant complex structure $J$. If it carries an lcK metric $\omega_1$ and a pluriclosed metric $\omega_2$, both of them compatible with the complex structure $J$, then $(M, J)$ is a complex torus.
\end{corollary}

\begin{proof}
	One can again use an averaging procedure for $\omega_2$ to obtain an invariant pluriclosed metric $\omega_2$ (see \cite[Proposition 21]{u07}).
\end{proof}

\begin{remark}
	\ref{thm:lckkG} does not hold true if we drop the assumption of left-invariance for the $k$-Gauduchon metric, as the following example shows:
	
	Take the nilmanifold $M$ in complex dimension $3$ satisfying equations \eqref{eq:lckbal1} and let $\omega$ be the left-invariant lcK form. According to \cite[Proposition 3.8]{ip}, the associated $\gamma_1(\omega)$ invariant introduced in \cite{fww13} is negative. On the other hand, since $M$ is a $3$-fold, there exists a Hermitian metric $\omega'$ with positive $\gamma_1(\omega')$ (see \cite[Theorem 6]{fww13}). By \cite[Corollary 10]{fww13}, there exists a metric $\omega_0$ with $\gamma_1(\omega_0) = 0$. Due to \cite[Proposition 8]{fww13}, we may find a $1$-Gauduchon metric in the conformal class of $\omega_0$.
\end{remark}

\begin{remark} The balanced and $k$-Gauduchon conditions are however compatible on the same nilmanifold, as long as they are not imposed to the same metric. In \cite{lu17} a family of compact complex nilmanifolds carrying both a balanced and an astheno-K\" ahler metric (which is also $k$-Gauduchon) is constructed (see \cite[Theorem 2.4]{lu17}). See another explicit example of such a phenomenon on a complex nilmanifold of dimension 4 in \cite[Proposition 4.8, Corollary 4.10]{fu11}. 
\end{remark}

\subsubsection{Locally conformally hyperK\" ahler structures versus left-invariance}

We start with  the following general statement concerning lcK complex nilmanifolds:

\begin{theorem}  On a compact complex nilmanifold $(M=\Gamma\backslash G, J)$ with  left-invariant $J$, there exists no left-invariant non-degenerate $(2,0)$-form $\omega$ such that $d\omega=\theta \wedge \omega$, where $\theta$ is the Lee form of a left-invariant lcK metric. 
\end{theorem}

\begin{proof}
	In order to prove this, we use again Sawai's description given by \eqref{eq:Liebrackets} and \eqref{eq:lckbal1}. The corresponding co-frame of $\mathfrak{g}^{1,0}$ is given by:
\begin{equation}\label{lck10}
\left\{
    \begin{array}{ll}
       d\alpha_i=0, &  1 \leq i \leq 2m-1, \\[.1in]
       d\alpha_{2m}=\displaystyle\frac{1}{2}\sum_{i=1}^{2m-1} \alpha_i \wedge \overline{\alpha}_i
    \end{array}
\right.
\end{equation} 
Let 
$$L(M, J):=\{\theta \mid \text{there exists  $\omega$ left-invariant and lcK:\ $d\omega=\theta \wedge \omega$}\}$$
be the space of those 1-forms which can be Lee forms of an lcK metric. 
Since by \eqref{Lee}, the Lee form $\theta$ of a left-invariant lcK metric is left-invariant, we claim that:
\begin{equation}
L(M, J)=\{t\beta_{2m}+t_1\beta_1+\ldots t_{2m-1}\beta_{2m-1} \mid t >0,\ t_i \in \mathbb{R}, \ \beta_i:=\alpha_i + \overline{\alpha}_i\},
\end{equation}
This is a consequence of the fact the positive $(1,1)$-form 
$$\omega_0:= \sum_m\beta_{2m} \wedge J\beta_{2m} - dJ\beta_{2m}$$ 
is a Vaisman metric with Lee form $\beta_{2m}$ (see \cite{s07}) and by the following result of Tsukada:
\begin{theorem}\noindent{\rm(}\cite[Theorem 5.1]{tsu}\noindent{\rm)}
Let $(M, J, g)$ be a compact Vaisman manifold and $\theta_0$ the the Lee form of $g$. The set of possible cohomology classes of Lee forms for lcK metrics is:
\begin{equation}
\mathrm{Lee}(M, J):=\{t[\theta_0]+ \mathcal{H} \mid t>0\},
\end{equation} 
where $ \mathcal{H} := \{[\eta] \mid \eta=\Re \gamma,\ \gamma\, \ \text{holomorphic} \ $1$-\text{form},\ \  d\eta=0\}\subseteq H^1_{dR}(M) $.
\end{theorem}

\medskip

Let now $\omega$ be a non-degenerate $(2,0)$-form such that $d\omega=\theta \wedge \omega$, with 
$$\theta=t\beta_{2m}+t_1\beta_1+\ldots t_{2m-1}\beta_{2m-1},\quad t>0.$$ 
Then $\omega=\sum_{i<j} A_{ij} \alpha_i \wedge \alpha_j$, with $A_{ij} \in \mathbb{C}$.

By \eqref{lck10}, we thus obtain 
\begin{equation}
d\omega=\sum_{i \neq j=1}^{2m-1}A_{i2m}\alpha_i \wedge \alpha_j \wedge \overline{\alpha}_j=\theta \wedge \sum_{i<j} A_{ij} \alpha_i \wedge \alpha_j. 
\end{equation}

Since $\omega$ is non-degenerate, there exists an index $1 \leq i_0 \leq 2m-1$ such that $A_{i_02m} \neq 0$. As $t>0$, in the right hand side term the monomial $\overline{\alpha}_{2m} \wedge \alpha_{i_0} \wedge \alpha_{2m}$ appears with positive coefficient, whereas in the left hand side it does not and the claim is proved. \end{proof}

\medskip

Recall now that a hyperHermitian manifold $(M,g,I,J,K)$ is {\em locally conformally hyperK\"ahler (lchK)} if $g$ is locally conformal to a hyperk\"ahler metric. The metric $g$ is then lcK w.r.t. to all of $I,J,K$, with the same Lee form: denoting with $\omega_I$ (respectively $\omega_J$, $\omega_K$) the fundamental form asssociated to the Hermitian structure $(g,I)$ (respectively $(g,J)$, $(g,K)$), we then have $d\omega_I=\theta\wedge\omega_I$, $d\omega_J=\theta\wedge\omega_J$, and $d\omega_K=\theta\wedge\omega_K$. If $M$ is compact, such a metric is necessarily Vaisman (see \cite{oo} for an alternative characterization). 

We  prove that the locally conformally hyperK\" ahler condition is incompatible with left-invariance. Indeed, an lchK structure with left-invariant structures $I_1$, $I_2$, $I_3$ would imply the existence of the left-invariant (2, 0)-form $\tilde{\omega}:=\omega_2 + \mathrm{i} \omega_3$, with respect to $I$, and moreover it would satisfy $d\tilde{\omega}=\theta \wedge \tilde{\omega}$, which is impossible. We thus have:

\begin{corollary} Let $\Gamma \backslash G$ be a compact nilmanifold and $g$ a left-invariant metric. There are no left-invariant complex structures $I, J, K$ such that $g$ becomes lchK with respect to them.
\end{corollary} 

\hfill

\noindent{\bf Acknowledgment:} We thank Anna Fino and Andrei Moroianu for a careful reading of a first draft of the paper, Stefan Ivanov for pointing out the results in \cite{ip}, Luigi Vezzoni for useful suggestions and the anonymous referee for his or her helpful comments.

\end{document}